\documentclass[11pt]{article}
\usepackage[UKenglish]{isodate}

\usepackage[T1]{fontenc}
\usepackage[noBBpl]{mathpazo}

\usepackage{amsmath, amssymb, amsthm, amsfonts, dsfont, zlmtt, graphicx, float, tikz, url, xpatch}

\usepackage[a4paper, margin=2.8cm]{geometry}
\usepackage{microtype}
\usepackage[shortcuts]{extdash}
\usepackage{titlesec}
\titlespacing*{\section}{0pt}{\baselineskip}{0pt}
\titlespacing*{\subsection}{0pt}{0.5\baselineskip}{0pt}

\let\OLDthebibliography\thebibliography
\renewcommand\thebibliography[1]{
  \OLDthebibliography{#1}
  \setlength{\parskip}{0pt}
  \setlength{\itemsep}{4pt plus 0.3ex}
}

\usepackage[shortlabels]{enumitem}
\setlist{leftmargin=0.8cm,topsep=0pt}
\setlist[enumerate]{label=\rm{(\roman*)}}

\numberwithin{equation}{section}

\makeatletter
\def\blfootnote{\gdef\@thefnmark{}\@footnotetext}
\makeatother
\newcommand{\dateline}{\enlargethispage{16.5pt}\blfootnote{\emph{Date} \today}}

\newcommand{\<}{\langle}
\renewcommand{\>}{\rangle}
\renewcommand{\leq}{\leqslant}

\renewcommand{\geq}{\geqslant}

\newtheoremstyle{shdefinition}{\topsep}{0.4\topsep}{}{}{\bfseries}{.}{0.5em}{} 
\newtheoremstyle{shplain}{\topsep}{0.4\topsep}{\itshape}{}{\bfseries}{.}{0.5em}{} 

\theoremstyle{shdefinition}
\newtheorem{definition}{Definition}[section]
\newtheorem*{definition*}{Definition}

\newtheorem{remark}[definition]{Remark}
\newtheorem{myquestion}{Question}
\newtheorem{example}[definition]{Example}
\newtheorem{notation}[definition]{Notation}

\theoremstyle{shplain}

\newtheorem*{conjecture*}{Conjecture}

\newtheorem{theorem}[definition]{Theorem}
\newtheorem*{theorem*}{Theorem}
\newtheorem{mytheorem}{Theorem}
\newtheorem{proposition}[definition]{Proposition}
\newtheorem{lemma}[definition]{Lemma}
\newtheorem{corollary}[definition]{Corollary}

\begin{document}

\begin{center}
{\LARGE \textbf{Infinite $\boldsymbol{\frac{3}{2}}$-generated groups}} \\[11pt]
{\Large Casey Donoven and Scott Harper} \normalsize \\[22pt]
\end{center}

\begin{center}
\begin{minipage}{0.8\textwidth}
{\small \textbf{Abstract.} Every finite simple group can be generated by two elements, and Guralnick and Kantor proved that, moreover, every nontrivial element is contained in a generating pair. Groups with this property are said to be $\frac{3}{2}$\=/generated. Thompson's group $V$ was the first finitely presented infinite simple group to be discovered. The Higman--Thompson groups $V_n$ and the Brin--Thompson groups $mV$ are two families of finitely presented groups that generalise $V$. In this paper, we prove that all of the groups $V_n$, $V_n'$ and $mV$ are $\frac{3}{2}$\=/generated. As far as the authors are aware, the only previously known examples of infinite noncyclic $\frac{3}{2}$\=/generated groups are the pathological Tarski monsters. We conclude with several open questions motivated by our results. \par
}
\end{minipage}
\end{center}

\dateline

\section{Introduction}

A group $G$ is \emph{$d$\=/generated} if it has a generating set of size $d$. We say that $G$ is \emph{$\frac{3}{2}$\=/generated} if $G$ is $2$\=/generated and, moreover, every nontrivial element of $G$ is contained in a generating pair. 

It is a remarkable fact that every finite simple group is $2$\=/generated. This result relies on the Classification of Finite Simple Groups and the proof for the groups of Lie type was given by Steinberg in 1962 \cite{ref:Steinberg62}. In that paper, Steinberg asked whether every finite simple group has the much stronger property of being $\frac{3}{2}$-generated. The alternating (and symmetric) groups of degree at least five had been known to be $\frac{3}{2}$-generated since the work of Piccard in 1939 \cite{ref:Piccard39}. In 2000, Guralnick and Kantor \cite{ref:GuralnickKantor00} resolved this long-standing question of Steinberg and proved the following.

\begin{theorem*}
Every finite simple group is $\frac{3}{2}$\=/generated.
\end{theorem*} 

This theorem has motivated a great deal of recent work on the generation of finite groups (see Section~3 of Burness' recent survey article \cite{ref:Burness19} for a good overview of this work).

It is easily seen that if $G$ is a $\frac{3}{2}$\=/generated group, then every proper quotient of $G$ is cyclic. The following conjecture, due to Breuer, Guralnick and Kantor \cite{ref:BreuerGuralnickKantor08}, avers that this clearly necessary condition is indeed sufficient for finite groups.

\begin{conjecture*}
A finite group $G$ is $\frac{3}{2}$\=/generated if and only if every proper quotient of $G$ is cyclic.
\end{conjecture*}

For recent work towards proving this conjecture, see \cite{ref:BurnessGuest13, ref:Harper17}.

The above necessary condition for $\frac{3}{2}$\=/generation is \emph{not} sufficient for infinite groups. Indeed, the infinite alternating group $A_\infty$ is simple but not finitely generated, let alone $\frac{3}{2}$\=/generated. Moreover, in \cite{ref:Guba82} a simple group is exhibited that is finitely generated but not 2-generated, which solves Problem~6.44 in the Kourovka Notebook \cite{ref:Kourovka14}. However, the authors are not aware of any $2$\=/generated group with no noncyclic proper quotients that is not $\frac{3}{2}$\=/generated. Equally, the authors are not aware of any existing examples of noncyclic infinite $\frac{3}{2}$\=/generated groups, other than \emph{Tarski monsters}. These latter groups are the infinite groups whose only proper nontrivial subgroups have order $p$ for a fixed prime $p$; they are clearly simple and $\frac{3}{2}$\=/generated, and they were proved to exist for all $p > 10^{75}$ by Olshanskii in \cite{ref:Olshanskii80}. 

In this paper, we provide two natural infinite families of infinite $\frac{3}{2}$\=/generated groups.

Thompson's group~$V$, introduced in 1965, was the first known example of a finitely presented infinite simple group \cite{ref:Thompson65}, and until 2000 \cite{ref:BurgerMozes00} all known examples of such groups were close relatives of $V$. The group $V$ naturally arises as a group of homeomorphisms of Cantor space $\mathfrak{C} = \{0,1\}^{\mathbb{N}}$, and, together with its subgroup $T$ and $F$ and the generalisations we will soon introduce, it has been the subject of much research, finding connections with dynamics and the word problem, among much else (see \cite{ref:BleakCameronAl16,ref:BleakMatucciNeunhoffer16,ref:BleakSalazarDiaz13,ref:CannonFloydParry96,ref:Thumann17} for example). 

As observed by Mason \cite{ref:Mason77}, $V$ is 2-generated, and we prove the following.

\begin{theorem*}
Thompson's group $V$ is $\frac{3}{2}$-generated.
\end{theorem*}

In fact, we will prove two theorems which generalise the above theorem in two different directions.

The \emph{Higman--Thompson group} $V_n$, for all $n \geq 2$, is an infinite finitely presented group, which was introduced by Higman in \cite{ref:Higman74}. The group $V_n$ has a natural action on $n$\=/ary Cantor space $\mathfrak{C}_n = \{0,1,\dots,n-1\}^{\mathbb{N}}$, and the group $V_2$ is nothing other than Thompson's group $V$. The derived subgroup of $V_n$ equals $V_n$ when $n$ is even and has index two when $n$ is odd. In both cases, $V_n'$ is simple. Mason proved that both $V_n$ and $V_n'$ are $2$\=/generated \cite{ref:Mason77}, and we prove the following.

\begin{mytheorem}\label{thm:higman}
For all $n \geq 2$, the Higman--Thompson groups $V_n$ and $V_n'$ are $\frac{3}{2}$-generated.
\end{mytheorem}

In particular, the groups $V_n$ when $n$ is odd give examples of infinite noncyclic $\frac{3}{2}$\=/generated groups that are not simple.

The \emph{Brin--Thompson group} $mV$, for all $m \geq 1$, acts on $\mathfrak{C}^m$ and was defined by Brin in \cite{ref:Brin04}. The groups $V=1V, 2V, 3V, \dots$ are pairwise nonisomorphic \cite{ref:BleakLanoue10}, simple \cite{ref:Brin10} and $2$\=/generated \cite[Corollary~1.3]{ref:Quick19}.

\begin{mytheorem}\label{thm:brin}
For all $m \geq 1$, the Brin--Thompson group $mV$ is $\frac{3}{2}$-generated.
\end{mytheorem}

Our proofs owe a great deal to the \emph{Bleak--Quick perspective}, introduced in \cite{ref:BleakQuick17}, that $mV$ (especially $V$) is analogous to the finite symmetric group, a viewpoint that readily extends to $V_n$. Indeed, in a sense, our proofs highlight that the $\frac{3}{2}$\=/generation of these groups is analogous to the $\frac{3}{2}$\=/generation of the finite symmetric groups.

For instance, the groups $V_n$ and $mV$ are generated by an infinite conjugacy class of involutions, known as \emph{transpositions}, so called because of their similarity to the transpositions in the finite symmetric group. In $mV$ and $V_{2k}$ every element is a product of an even number of transpositions (even a transposition is such a product), so we think of these groups as analogous to both the alternating and symmetric groups. When $n$ is odd, $V_n'$ is the subgroup of $V_n$ containing the elements that are a product of an even number of transpositions.

Our proofs are concrete and constructive. Let $G$ be one of $V_n$, $V_n'$ or $mV$ (here and throughout, the reader is encouraged to think of $V$ in the first instance). For every nontrivial element $g \in G$ we will exhibit an element $h \in G$ of finite order and prove that $\<g,h\>=G$. We prove that $\<g,h\>=G$ by essentially showing that $\<g,h\>$ contains all of the transpositions. We modify our approach slightly for $V_n'$, which is not generated by the transpositions when $n$ is odd. To describe the element $h \in G$ we will make use of the permutation-like notation of Bleak and Quick. 

We will begin with some preliminary results in Sections~\ref{s:prelims} and~\ref{s:generators}, before turning to the proofs of Theorems~\ref{thm:higman} and~\ref{thm:brin} in Section~\ref{s:proofs}. The authors expect that collecting together uniformly stated results on the generating sets for $V_n$, $V_n'$ and $mV$ in Section~\ref{s:generators} will be useful for others interested in studying these groups, as this information is currently spread across the literature in terms of various different descriptions of and notation for these groups. In Section~\ref{s:questions}, we conclude by posing several natural questions that are raised by our theorems.

\subsection*{Acknowledgements}
The second author thanks the London Mathematical Society for financial support through an Early Career Fellowship. Both authors thank an anonymous referee for their helpful comments on a previous version of the paper. \enlargethispage{3pt}

\section{Preliminaries} \label{s:prelims}

\subsection{Higman--Thompson group $V_n$} \label{ss:p_higman}

Fix $n \geq 2$ and write $[n] = \{ 0,1,\dots,n-1 \}$. Let $X_n = [n]^*$ be the set of all finite words over $[n]$, and let $\mathfrak{C}_n = [n]^{\mathbb{N}}$ be the \emph{$n$\=/ary Cantor set}, the set of all infinite sequences over $[n]$. For $u = u_1u_2 \dots u_k \in X_n$ we write $|u| = k$. For $u \in X_n$, we write $u\mathfrak{C}_n = \{ uw \mid w \in \mathfrak{C}_n \}$. A finite set $A \subseteq X_n$ is a \emph{basis} of $\mathfrak{C}_n$ if $\{ u\mathfrak{C}_n \mid u \in A \}$ is a partition of $\mathfrak{C}_n$. (Various terms appear in the literature for what we call a basis.) 

The \emph{Higman--Thompson group} $V_n$ is the group of bijections $g \in \mathrm{Sym}(\mathfrak{C}_n)$ for which there exists a \emph{basis pair}, namely a bijection $\sigma\colon A \to B$ between two bases $A$ and $B$ of $\mathfrak{C}_n$ such that $(uw)g = (u\sigma)w$ for all $u \in A$ and all $w \in \mathfrak{C}_{n}$.

The following records several important results about $V_n$, which were proved by Higman when he introduced these groups \cite{ref:Higman74}.

\begin{theorem}
Let $n \geq 2$. The following hold:
\begin{enumerate}
\item $V_n$ is a finitely presented infinite group
\item $V_n$ is simple if $n$ is even, and the derived subgroup $V_n'$ has index two and is simple if $n$ is odd
\item if $k \geq 2$, then $V_n \cong V_k$ if and only if $n=k$.
\end{enumerate} 
\end{theorem}

\subsection{Brin--Thompson group $mV$} \label{ss:p_brin}

Write $\mathfrak{C} = \mathfrak{C}_2 = \{0,1\}^{\mathbb{N}}$ and $X = X_2 = \{0,1\}^*$. Fix $m \geq 1$ and let $u = (u_1,\dots,u_m) \in X^m$. We write $|u| = \min\{ |u_i| \mid 1 \leq i \leq m \}$. For an element $w \in (w_1,\dots,w_m)$ of either $X^m$ or $\mathfrak{C}^m$, we write $uw = (u_1w_1,\dots,u_mw_m)$. In addition, we write $u\mathfrak{C}^m = \{ uw \mid w \in \mathfrak{C}^m \}$. A finite set $A \subseteq X^m$ is a \emph{basis} of $\mathfrak{C}^m$ if $\{ u\mathfrak{C}^m \mid u \in A\}$ is a partition of $\mathfrak{C}^m$. 

The \emph{Brin--Thompson group} $mV$ is the group of bijections $g \in \mathrm{Sym}(\mathfrak{C}^m)$ for which there exists a \emph{basis pair}, namely a bijection $\sigma\colon A \to B$ between two bases $A$ and $B$ of $\mathfrak{C}^m$ such that $(uw)g = (u\sigma)w$ for all $u \in A$ and all $w \in \mathfrak{C}^m$.

Part~(i) of the following was proved by Brin \cite{ref:Brin10}, and (ii) is due to Bleak and Lanoue \cite{ref:BleakLanoue10}.

\begin{theorem}
Let $m \geq 1$. The following hold:
\begin{enumerate}
\item $mV$ is a finitely presented infinite simple group
\item if $k \geq 1$, then $mV \cong kV$ if and only if $m=k$.
\end{enumerate} 
\end{theorem}

\subsection{Thompson's group $V$} \label{ss:p_thompson}

The special case of $V = V_2 = 1V$ should shed some light on the general definitions above. We may identify the elements of $X = \{0,1\}^*$ with the vertices of the infinite binary rooted tree (where, when drawn in the plane, $0$ indicates ``go left'' and $1$ indicates ``go right''). From this viewpoint, a basis of $\mathfrak{C}$ is exactly the set of leaves of a finite binary rooted tree. Therefore, the basis pairs give the familiar \emph{tree pair} notation for elements of $V$, namely $(A, \sigma, B)$ where $A$ and $B$ are finite binary rooted trees and $\sigma$ is a bijection between the leaves of $A$ and $B$. When referring to tree pairs, we will number the leaves of the trees to manifest the bijection and we will adopt the convention that unlabelled leaves are fixed. 

Three elements $\alpha,\beta,\gamma \in V$ are represented by tree pairs in Figure~\ref{fig:tree_pairs}. To demonstrate the action of these elements on $\mathfrak{C}$, note for example that
\begin{gather*}
\overline{10}\alpha = {10}1010\ldots \alpha = {110}1010\ldots = 110\overline{10} \\
\overline{10}\beta  = {101}010\ldots \beta  = {110}010\ldots  = 110\overline{01} \\ 
\overline{10}\gamma = {1}01010\ldots \gamma = {11}01010\ldots = 11\overline{01}
\end{gather*}

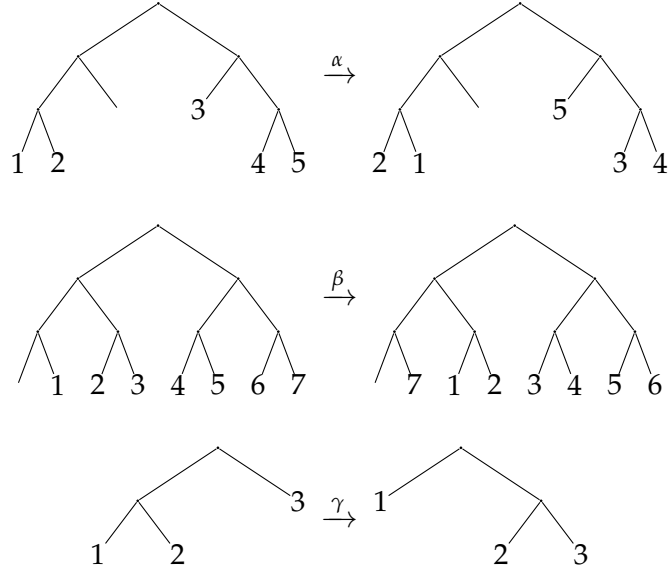
\begin{figure}
  \[
  \begin{tikzpicture}[
      inner sep=0pt,
      baseline=-30pt,
      level distance=20pt,
      level 1/.style={sibling distance=60pt},
      level 2/.style={sibling distance=30pt},
      level 3/.style={sibling distance=15pt}
    ]
    \node (root) [circle,fill] {}
    child {node (0) [circle,fill] {}
      child {node (00) [circle,fill] {}
        child {node (000) {$1$}}
        child {node (001) {$2$}}}
      child {node (01) {}}}
    child {node (1) [circle,fill] {}
      child {node (10) {$3$}}
      child {node (11) [circle,fill] {}
        child {node (110) {$4$}}
        child {node (111) {$5$}}}};
  \end{tikzpicture}
  \ \xrightarrow{\,\alpha\,} \
  \begin{tikzpicture}[
      inner sep=0pt,
      baseline=-30pt,
      level distance=20pt,
      level 1/.style={sibling distance=60pt},
      level 2/.style={sibling distance=30pt},
      level 3/.style={sibling distance=15pt}
    ]
    \node (root) [circle,fill] {}
    child {node (0) [circle,fill] {}
      child {node (00) [circle,fill] {}
        child {node (000) {$2$}}
        child {node (001) {$1$}}}
      child {node (01) {}}}
    child {node (1) [circle,fill] {}
      child {node (10) {$5$}}
      child {node (11) [circle,fill] {}
        child {node (110) {$3$}}
        child {node (111) {$4$}}}};
  \end{tikzpicture}
\]
\vspace{11pt}
\[
  \begin{tikzpicture}[
        inner sep=0pt,
        baseline=-30pt,
        level distance=20pt,
        level 1/.style={sibling distance=60pt},
        level 2/.style={sibling distance=30pt},
        level 3/.style={sibling distance=15pt}
      ]
      \node (root) [circle,fill] {}
      child {node (0) [circle,fill] {}
        child {node (00) [circle,fill] {}
          child {node (000) {}}
          child {node (001) {$1$}}}
        child {node (01) [circle,fill] {}
          child {node (010) {$2$}}
          child {node (011) {$3$}}}}
      child {node (1) [circle,fill] {}
        child {node (10) [circle,fill] {}
          child {node (100) {$4$}}
          child {node (101) {$5$}}}
        child {node (11) [circle,fill] {}
          child {node (110) {$6$}}
          child {node (111) {$7$}}}};
    \end{tikzpicture}
    \ \xrightarrow{\,\beta\,} \
  \begin{tikzpicture}[
        inner sep=0pt,
        baseline=-30pt,
        level distance=20pt,
        level 1/.style={sibling distance=60pt},
        level 2/.style={sibling distance=30pt},
        level 3/.style={sibling distance=15pt}
      ]
      \node (root) [circle,fill] {}
      child {node (0) [circle,fill] {}
        child {node (00) [circle,fill] {}
          child {node (000) {}}
          child {node (001) {$7$}}}
        child {node (01) [circle,fill] {}
          child {node (010) {$1$}}
          child {node (011) {$2$}}}}
      child {node (1) [circle,fill] {}
        child {node (10) [circle,fill] {}
          child {node (100) {$3$}}
          child {node (101) {$4$}}}
        child {node (11) [circle,fill] {}
          child {node (110) {$5$}}
          child {node (111) {$6$}}}};
    \end{tikzpicture}
  \]
\vspace{11pt}
\[ 
  \begin{tikzpicture}[
        inner sep=0pt,
        baseline=-30pt,
        level distance=20pt,
        level 1/.style={sibling distance=60pt},
        level 2/.style={sibling distance=30pt}
      ]
      \node (root) [circle,fill] {}
      child {node (0) [circle,fill] {}
        child {node (00) {$1$}}
        child {node (01) {$2$}}}
      child {node (1) {$3$}};
    \end{tikzpicture}
    \ \xrightarrow{\,\gamma\,} \
  \begin{tikzpicture}[
        inner sep=0pt,
        baseline=-30pt,
        level distance=20pt,
        level 1/.style={sibling distance=60pt},
        level 2/.style={sibling distance=30pt}
      ]
      \node (root) [circle,fill] {}
      child {node (0) {$1$}}
      child {node (1) [circle,fill] {}
        child {node (10) {$2$}}
        child {node (11) {$3$}}};
    \end{tikzpicture}
  \]
\caption{Three elements of $V$ represented as tree pairs} \label{fig:tree_pairs}
\end{figure}

\subsection{Permutations} \label{ss:p_permutations}

Let $G$ be $V_n$ or $mV$. An element $g \in G$ is a \emph{permutation} if it admits a basis pair $(A,\sigma,A)$. Following Bleak and Quick \cite{ref:BleakQuick17}, in this case, we identify $g \in G$ with $\sigma \in \mathrm{Sym}(A)$. For example, the elements $\alpha, \beta \in V$ given in Figure~\ref{fig:tree_pairs} are permutations and we write
\[
\alpha = (000 \ \ 001)(10 \ \ 110 \ \ 111) \qquad \beta = (001 \ \ 010 \ \ 011 \ \ 100 \ \ 101 \ \ 110 \ \ 111).
\]
A permutation $g \in G$ is a \emph{transposition} if it admits a basis pair $(A,\sigma,A)$ where $\sigma \in \mathrm{Sym}(A)$ is a transposition. Notice that $\alpha^3 = (000 \ \ 001)$ is a transposition.

\begin{remark}
Although we do not make use of this fact, we note that the permutations are exactly the elements of $G$ of finite order. For $V_n$, this is \cite[Lemma~17]{ref:Scott92}, and the same proof applies to $mV$.
\end{remark}

\subsection{Incomparability} \label{ss:p_incomparability}

Let $(G,\Gamma,\Omega)$ be $(V_n,\mathfrak{C}_n,X_n)$ or $(mV,\mathfrak{C}^m,X^m)$. We say that $u,v \in \Omega$ are \emph{incomparable} and write $u \perp v$ if $u\Gamma$ and $v\Gamma$ are disjoint. The following two technical lemmas are straightforward.

\begin{lemma} \label{lem:incomparables_characterisation} \quad
\begin{enumerate}
\item Let $u,v \in X_n$. Then $u \perp v$ if and only if neither $u$ is a prefix of $v$ nor $v$ is a prefix of $u$.
\item Let $u=(u_1,\dots,u_m)$ and $v=(v_1,\dots,v_m)$ be elements of $X^m$. Then $u \perp v$ if and only if there exists $1 \leq i \leq m$ for which $u_i \perp v_i$.
\end{enumerate}
\end{lemma}

\begin{lemma} \label{lem:incomparables_extend}
For any set $A$ of pairwise incomparable elements of $\Omega$, there exists a basis $B$ such that $A \subseteq B$.
\end{lemma}

\begin{lemma} \label{lem:incomparables_exist}
Let $g \in G$ be nontrivial. Then there exists $u \in \Omega$ such that $u \perp ug$.
\end{lemma}

\begin{proof}
We assume that $G = mV$ since the proof is similar and easier when $G = V_n$. Let $(A,\sigma,B)$ be a basis pair for $g$ and let $u = (u_1,\dots,u_m) \in A$ such that $u \neq ug$. Write $ug = v = (v_1,\dots,v_m)$. If $u \perp v$, then the proof is complete. Therefore, assume that $u$ and $v$ are not incomparable. Consequently, there exists $1 \leq i \leq m$ such that either $u_i$ is a proper prefix of $v_i$ or $v_i$ is a proper prefix of $u_i$ (otherwise $u=v$). We will assume that $u_i$ is a proper prefix of $v_i$; the proof works in the other case making the obvious changes. Write $v_i = u_iaw$ for $a \in \{0,1\}$ and $w \in \{0,1\}^*$. Let $b\in\{0,1\}$ be such that $a\neq b$. Define $u' = (u_1,\dots,u_i b,\dots,u_m)$. Now the $i$th coordinate of $u'g$ satisfies $(u'g)_i = v_i b = u_iaw b$. Since $|(u'g)_i| > |u_ia| = |u'_i|$ we know that $(u'g)_i$ is not a prefix of $u'_i$, and visibly $u'_i = u_i b$ is not a prefix of $(u'g)_i = u_iaw b$. Therefore, $u'$ and $u'g$ are incomparable, as required.
\end{proof}

\subsection{Convenient shorthand} \label{ss:p_notation}

Let $(G,\Gamma,\Omega)$ be $(V_n,\mathfrak{C}_n,X_n)$ or $(mV,\mathfrak{C}^m,X^m)$. For $g \in G$ and $u \in \Omega$, define an element $g_{[u]} \in G$ as follows: for each $w \in \Gamma$, 
\[
w g_{[u]} = \left\{
\begin{array}{ll}
u (v g)   & \text{if $w=uv$} \\
w         & \text{if $w \not\in u\Gamma$.} \\
\end{array}
\right.
\] 
Write 
\[
G_{[u]} = \{ g_{[u]} \mid g \in G \}, 
\]
noting that $G_{[u]}$ acts on $u\Gamma$ and $G_{[u]} \cong G$. Observe that if $g = (v_{11} \ \ \cdots \ \ v_{1k_1}) \cdots (v_{d1} \ \ \cdots \ \ v_{dk_d})$, then 
\[
g_{[u]} = (uv_{11} \ \ \cdots \ \ uv_{1k_1}) \cdots (uv_{d1} \ \ \cdots \ \ uv_{dk_d}).
\]

\subsection{Bertrand's postulate} \label{ss:p_bertrand}

In 1852, Chebyshev proved Bertrand's postulate: for all integers $k > 3$, there exists a prime number $p$ such that $k < p < 2k$. Indeed, by the Prime Number Theorem, for all $\varepsilon > 0$, there exists $k_0 > 0$ such that for all $k > k_0$, there exists a prime number $p$ such that $k < p < (1+\varepsilon)k$. We will make use of the following useful non-asymptotic bound due to Nagura \cite{ref:Nagura52}.

\begin{theorem} \label{thm:bertand}
Let $k \geq 25$. Then there exists a prime $p$ such that $k < p < \frac{6}{5}k$.
\end{theorem}

\subsection{Generating symmetric and alternating groups} \label{ss:p_generating}

The following is \cite[Lemma~2]{ref:Mason77}; our proof is based on, but shorter than, the original. Here, and throughout, for group elements $x$ and $y$ we write $x^y = y^{-1}xy$ and $[x,y] = x^{-1}y^{-1}xy$.

\begin{lemma} \label{lem:permutations_fact}
Let $n \geq 7$ and $1 < a \leq b < n$. Then $A_n \leq \< (1 \ \ 2 \ \ \cdots \ \ b), (a \ \ a+1 \ \ \cdots \ \ n) \>$.
\end{lemma}

\begin{proof}
To simplify the details of the proof, let us assume that $n \geq 9$; the result can be easily verified (for example, in \textsc{Magma} \cite{ref:Magma}) when $n \in \{7,8\}$. Write $x = (1 \ \ 2 \ \ \cdots \ \ b)$ and $y = (a \ \ a+1 \ \ \cdots \ \ n)$. Since $b \geq a$ the subgroup $\<x,y\> \leq S_n$ is transitive. For all $0 < j < n-b$, the element $x^{y^j} = (1 \ \ \cdots \ \ a-1 \ \ a+j \ \ \cdots \ \ b+j)$ fixes $n$, so the stabiliser of $n$ in $\<x,y\>$ is transitive. Therefore, $\<x,y\>$ is $2$\=/transitive. Now $[x,y]$ is $(1 \ \ b+1)(a \ \ a+1)$ if $a < b$ and $(1 \ \ a \ \ a+1)$ if $a=b$. Therefore, $\<x,y\>$ is a primitive group of degree $n \geq 9$ containing an element supported on at most four points, so $\<x,y\> \geq A_n$ (see \cite[Example~3.3.1]{ref:DixonMortimer96}). 
\end{proof}

\section{Generating Sets} \label{s:generators}

This section is dedicated to recording important generating sets of the Higman--Thompson and Brin--Thompson groups. To aid the reader, let us first comment on the prototype $V$.

\subsection{Thompson's group} \label{ss:generators_v}

As Brin records in \cite[Lemma~12.2]{ref:Brin04}
\begin{equation}
V = \< \{ (u \ \ v) \mid \text{$u,v \in X$ and $u \perp v$} \} \>. \label{eq:transpositions}
\end{equation}
Indeed he uses this observation to provide a concise proof of the simplicity of $V$ analogous to the usual proof of the simplicity of the alternating group. In \cite[Theorem~1.1]{ref:BleakQuick17}, Bleak and Quick obtain an elegant presentation for $V$ whose generators are the transpositions, which reflects both the Coxeter presentation of the symmetric group but also the self-similarity that $V$ inherits from the Cantor set. 

Let $X^{(3)} = \{ 000, 001, 010, 011, 100, 101, 110, 111 \}$ and let $S^{(3)} = \mathrm{Sym}(X^{(3)}) \cong S_8$. We will see in Proposition~\ref{prop:delta_symmetric} that
\begin{equation}
V = \< (10 \ \ 110 \ \ 111), S^{(3)}\>. \label{eq:delta_symmetric}
\end{equation}
(It is not difficult to derive the generating set in \eqref{eq:delta_symmetric} from a similar one given in \cite{ref:Mason77}.) Now
\[
S^{(3)} = \< (000 \ \ 001), (001 \ \ 010 \ \ 011 \ \ 100 \ \ 101 \ \ 110 \ \ 111) \>.
\] 
Therefore, since $(000 \ \ 001)$ and $(10 \ \ 110 \ \ 111)$ have coprime orders and commute, we conclude that 
\begin{equation}
V = \< \alpha, \beta\> \label{eq:alpha_beta}
\end{equation} 
where $\alpha = (000 \ \ 001)(10 \ \ 110 \ \ 111)$ and $\beta = (001 \ \ 010 \ \ 011 \ \ 100 \ \ 101 \ \ 110 \ \ 111)$ (see Figure~\ref{fig:tree_pairs}).

In the rest of this section, we show that $V_n$ and $mV$ have generating sets that reminiscent of those in \eqref{eq:transpositions}--\eqref{eq:alpha_beta}. These will play an important role in our proofs in Section~\ref{s:proofs}.

\subsection{Higman--Thompson groups} \label{ss:generators_higman}

Fix $n \geq 2$. The first result extends Brin's observation that $V$ is generated by transpositions.

\begin{lemma} \label{lem:transpositions}
The group $V_n$ is generated by the set of transpositions.
\end{lemma}

This is proved in the proof of \cite[Theorem~8.7]{ref:GarncarekLazarovich18}. Indeed, more is shown in that proof. If $n$ is even, then $(u \ \ v) = \prod_{a=0}^{n-1}(ua \ \ va)$, so every element of $V_n$ is a product of an even number of transpositions. In contrast, if $n$ is odd, then a transposition cannot be written as the product of an even number of transpositions, from which we deduce the following.

\begin{lemma} \label{lem:transpositions_derived}
Let $n$ be odd. The subgroup $V_n'$ consists of the elements of $V_n$ that can be written as a product of an even number of transpositions.
\end{lemma}

\begin{remark} \label{rem:transpositions}
Let $n$ be odd. In the proof of \cite[Theorem~8.7]{ref:GarncarekLazarovich18}, it is observed that $V_n'$ is generated by the set of \emph{double transpositions}, that is elements $(u_1 \ \ v_1)(u_2 \ \ v_2)$ where $u_1,v_1,u_2,v_2$ are pairwise incomparable. Let us explain how this can be deduced from Lemma~\ref{lem:transpositions_derived}. 

For $k \geq 0$, write $X^{(k)}_n = \{ x \in X_n \mid |x| = k \}$. First we claim that every product of two transpositions is a product of some number of double transpositions. Let $u_1,v_1,u_2,v_2 \in X_n$ such that $u_1 \perp v_1$ and $u_2 \perp v_2$. Observe that there exist $k \geq 0$ and $x_1,x_2 \in X_{n}^{(k)}$ such that $u_1x_1, v_1x_1, u_2x_2, v_2x_2$ are pairwise incomparable. We may write
\[
(u_1 \ \ v_1)(u_2 \ \ v_2) = \left( \prod_{x \in X_n^{(k)} \setminus \{ x_1 \}}(u_1x \ \ v_1x) \right) (u_1x_1 \ \ v_1x_1) (u_2x_2 \ \ v_2x_2) \left( \prod_{x \in X_n^{(k)} \setminus \{ x_2 \}}(u_2x \ \ v_2x) \right).
\]
Since $|X_n^{(k)} \setminus \{x_1\}| = |X_n^{(k)} \setminus \{x_2\}| = n^k-1$ is even, the above equation expresses the element $(u_1 \ \ v_1)(u_2 \ \ v_2)$ as a product of $n^k$ double transpositions. Therefore, every product of two transpositions is a product of double transpositions, so by Lemma~\ref{lem:transpositions_derived}, $V_n'$ is generated by the double transpositions. 

Since $V_n'$ is generated by the double transpositions, it is certainly generated by the even permutations in $V_n$. This implies that $V_n'$ is also generated by the \emph{three-cycles}, that is elements $(u \ \ v \ \ w)$, where $u$, $v$, $w$ are pairwise incomparable.
\end{remark}

It is via the following result that we use Lemmas~\ref{lem:transpositions} and~\ref{lem:transpositions_derived}.

\begin{lemma} \label{lem:transpositions_strong}
Let $u_0,u_1 \in X_n$ with $u_0 \neq u_1$ and $|u_0|=|u_1|=k \geq 2$. 
\begin{enumerate}
\item $V_n  = \< (u_0 \ \ v)         \mid \text{$u_0 \perp v$ and $|v| \geq k$}\>$
\item $V_n' = \< (u_0 \ \ u_1 \ \ v) \mid \text{$u_0 \perp v$, $u_1 \perp v$ and $|v| \geq k$}\>$ if $n$ is odd.
\end{enumerate} 
\end{lemma}

\begin{proof}
We begin with part~(i). Write
\[
H = \< \{ (u_0 \ \ v) \mid \text{$u_0 \perp v$ and $|v| \geq k$} \} \>.
\]
Let $v, w \in X_n$ such that $v \perp w$ and $|v|,|w| \geq k$. Write $v=v_1v_2$ and $w=w_1w_2$ where $|v_1|=|w_1|=k$. First assume that $u_0 \not\in \{v_1,w_1\}$, then
\[
(v \ \ w) = (u_0 \ \ v)^{(u_0 \ \ w)} \in H.
\] 
Now assume that $u_0 \in \{v_1,w_1\}$. Without loss of generality assume that $v_1=u_0$. If $v_1=w_1=u_0$, then
\[
(v \ \ w) = (u_0v_2 \ \ u_0w_2) = (u_1v_2 \ \ u_1w_2)^{(u_0 \ \ u_1)} = (u_0 \ \ u_1w_2)^{(u_0 \ \ u_1v_2)(u_0 \ \ u_1)} \in H.
\]
If $v_1 \neq w_1$, then fix $x \in X_n \setminus \{ v_1, w_1 \}$ with $|x|=k$ and note that
\[
(v \ \ w) = (u_0v_2 \ \ w_1w_2) = (xv_2 \ \ w_1w_2)^{(u_0 \ \ x)} = (u_0 \ \ w_1w_2)^{(u_0 \ \ xv_2)(u_0 \ \ x)} \in H.
\]
Therefore, in all cases, $(v \ \ w) \in H$. Since every transposition is a product of transpositions of the form $(v \ \ w)$ where $|v|, |w| \geq k$, we conclude, by Lemma~\ref{lem:transpositions}, that $H = V_n$.

We will now assume that $n$ is odd and consider part~(ii). Write
\[
K = \< \{ (u_0 \ \ u_1 \ \ v) \mid \text{$u_0 \perp v$, $u_1 \perp v$ and $|v| \geq k$ } \} \>.
\]
Let $v, w, z \in X_n$ be pairwise incomparable with $|v|,|w|,|z| \geq k$. Write $v=v_1v_2$, $w=w_1w_2$, $z=z_1z_2$, where $|v_1|=|w_1|=|z_1|=k$. First assume that $u_0,u_1 \not\in \{v_1,w_1,z_1\}$. Then 
\[
(v \ \ w \ \ z) = (u_0 \ \ u_1 \ \ z)^{(u_0 \ \ u_1 \ \ w)(u_0 \ \ u_1 \ \ v)} \in K.
\]

Next assume that exactly one of $u_0$ and $u_1$ is contained in $\{ v_1,z_1,w_1 \}$. Fix $x \in X_n \setminus \{u_0,u_1,v_1,w_1,z_1\}$ with $|x|=k$. Without loss of generality, we can assume that $u_0 = v_1$ and $u_1 \not\in \{v_1,z_1,w_1\}$. If $v_1=w_1=z_1$, then
\[
(v \ \ w \ \ z) = (xv_2 \ \ xw_2 \ \ xz_2)^{(u_0 \ \ u_1 \ \ x)} = (u_0 \ \ u_1 \ \ xz_2)^{(u_0 \ \ u_1 \ \ xw_2)(u_0 \ \ u_1 \ \ xv_2)(u_0 \ \ u_1 \ \ x)}.
\]
If $v_1 \not\in \{w_1,z_1\}$, then 
\[
(v \ \ w \ \ z) = (u_0v_2 \ \ w \ \ z) = (xv_2 \ \ w \ \ z)^{(u_0 \ \ u_1 \ \ x)} = (u_0 \ \ u_1 \ \ z)^{(u_0 \ \ u_1 \ \ w)(u_0 \ \ u_1 \ \ xv_2)(u_0 \ \ u_1 \ \ x)}.
\]
We may now assume, without loss of generality, that $v_1=w_1 \neq z_1$. In this case,
\[
(v \ \ w \ \ z) = (u_0v_2 \ \ u_0w_2 \ \ z) = (xv_2 \ \ xw_2 \ \ z)^{(u_0 \ \ u_1 \ \ x)} = (u_0 \ \ u_1 \ \ z)^{(u_0 \ \ u_1 \ \ xw_2)(u_0 \ \ u_1 \ \ xv_2)(u_0 \ \ u_1 \ \ x)}.
\]

Now assume that $u_0,u_1 \in \{v_1,w_1,z_1\}$. Fix $x,y \in X_n \setminus \{v_1,w_1,z_1\}$ with $|x|=|y|=k$. Without loss of generality, we can assume that $u_0 = v_1$ and $u_1 = w_1$. If $z_1 \not\in \{v_1,w_1\}$, then 
\begin{align*}
(v \ \ w \ \ z) &= (u_0v_2 \ \ u_1w_2 \ \ z) = (xv_2 \ \ yw_2 \ \ z)^{(u_0 \ \ u_1 \ \ y)(u_0 \ \ u_1 \ \ x)} \\
                &= (u_0 \ \ u_1 \ \ z)^{(u_0 \ \ u_1 \ \ yw_2)(u_0 \ \ u_1 \ \ xv_2)(u_0 \ \ u_1 \ \ y)(u_0 \ \ u_1 \ \ x)}.
\end{align*}
It remains to assume, without loss of generality, that $z_1=w_1$. In this final case,
\begin{align*}
(v \ \ w \ \ z) &= (u_0v_2 \ \ u_1w_2 \ \ u_1z_2) = (xv_2 \ \ yw_2 \ \ yz_2)^{(u_0 \ \ u_1 \ \ y)(u_0 \ \ u_1 \ \ x)} \\
                &= (u_0 \ \ u_1 \ \ yz_2)^{(u_0 \ \ u_1 \ \ yw_2)(u_0 \ \ u_1 \ \ xv_2)(u_0 \ \ u_1 \ \ y)(u_0 \ \ u_1 \ \ x)}.
\end{align*}
Therefore, in all cases $(v \ \ w \ \ z) \in K$. By Remark~\ref{rem:transpositions}, this implies that $K = V_n'$, as required.
\end{proof}

Recall that $X^{(k)}_n = \{ x \in X_n \mid |x| = k \}$, and write $S^{(k)}_n = \mathrm{Sym}(X^{(k)}_n)$ and $A^{(k)}_n = \mathrm{Alt}(X^{(k)}_n)$. Define
\begin{equation} \label{eq:delta}
\delta  = \prod_{a=0}^{n-1} (0 \ \ 1a) \quad \text{and} \quad \delta' = \delta(0 \ \ 2),
\end{equation}
where we only define $\delta'$ when $n$ is odd. In the following proposition, we use notation such as $\delta_{[u]}$, which we introduced in Section~\ref{ss:p_notation}.

\begin{proposition} \label{prop:delta_symmetric}
Let $u \in X_n$ with $|u| \geq 1$ and let $k = |u|+1$. Then
\begin{enumerate}
\item $V_n  = \< \delta_{[u]}, S^{(k)}_n \> = \< \delta_{[u]}, A^{(k+1)}_n \>$
\item $V_n' = \< \delta'_{[u]}, A^{(k)}_n \> = \< \delta'_{[u]}, A^{(k+1)}_n \>$ if $n$ is odd.
\end{enumerate}
\end{proposition}

\begin{proof}
For this proof, it will be convenient to write $X^{(k)}_n = \{ u_0, \dots, u_N \}$ where $N = n^k-1$. We will assume that $u_a = ua$ if $a \in [n]$. In addition, let $\tilde{\delta}$ be either $\delta_{[u]}$ or $\delta'_{[u]}$ and let $\tilde{S}$ be either $S^{(k)}_n$ or $A^{(k)}_n$.

Let $v \in X_n$ such that $v \perp u_{N-1}$, $v \perp u_N$ and $|v| \geq k$.  Assume that $\tilde{S} = S_2^{(2)}$ if $n=k=2$.

\noindent\textbf{Claim.} There exists $g \in \<\tilde{S}, \tilde{\delta} \>$ such that $u_{N-1}g=u_{N-1}$, $u_Ng = u_N$ and $u_1g = v$. 

\noindent\textit{Proof of Claim.} We induct on $|v|$. If $|v|=k$, then $v \in X^{(k)}_n$, so there is clearly an element $g \in \tilde{S}$ such that $u_{N-1}g=u_{N-1}$, $u_Ng = u_N$ and $u_1g=v$. Now let $l > k$ and assume that claim is true for all $v \in X_n$ such that $|v| < l$.  Assume $|v|=l$ and write $v=wa$ where $a \in [n]$. By induction, there exists $h \in \< \tilde{S}, \tilde{\delta} \>$ such that $u_{N-1}h = u_{N-1}$, $u_Nh=u_N$ and $u_1h=w$. First assume that $nk > 4$ and note that $u_1(u_2 \ \ u_1 \ \ u_0)\tilde{\delta}^{a+1} = u_1a$.  Therefore,
\[
v = wa = u_1ah = u_1(u_2 \ \ u_1 \ \ u_0)\tilde{\delta}^{a+1}h = u_1g,
\] 
where $g =(u_2 \ \ u_1 \ \ u_0)\tilde{\delta}^{a+1}h$ is an element of $\<\tilde{S},  \tilde{\delta}\>$ for which $u_{N-1}g = u_{N-1}$ and $u_Ng=u_N$. When $nk=4$, note that $\tilde{S} = S_2^{(2)}$ and, arguing as above, we see that $g = (u_1 \ \ u_0)\tilde{\delta}^{a+1}h \in \< \tilde{S}, \tilde{\delta}\>$ fixes $u_{N-1}$ and $u_N$ and satisfies $u_1g = v$. This proves the claim.

Let us now record some consequences of the claim.

Since $(u_N \ \ v) = (u_N \ \ u_1)^g \in \< S^{(k)}_n, \delta_{[u]} \>$, Lemma~\ref{lem:transpositions_strong}(i) implies that $\< S^{(k)}_n, \delta_{[u]} \> = V_n$. If $n$ is even, then $S^{(k)}_n \leq A^{(k+1)}_n$, so $\<A^{(k+1)}_n, \delta_{[u]}\> = V_n$. Now assume that $n$ is odd. In this case, $(u_{N-1} \ \ u_N \ \ v) = (u_{N-1} \ \ u_N \ \ u_1)^g \in \< A^{(k)}_n, \delta_{[u]}\>$, so Lemma~\ref{lem:transpositions_strong} implies that $V_n' \leq \< A^{(k)}_n, \delta_{[u]} \> \leq \< A^{(k+1)}_n, \delta_{[u]} \>$. Now $\delta_{[u]}$, being an $(n+1)$-cycle, is an odd permutation, so $\delta_{[u]} \not\in V_n'$ and again we conclude that $\< A^{(k+1)}_n, \delta_{[u]} \> = V_n$. This proves (i).

We now turn to (ii), where we assume that $n$ is odd. Similarly to above $(u_{N-1} \ \ u_N \ \ v) = (u_{N-1} \ \ u_N \ \ u_1)^g \in \< A^{(k)}_n, \delta'_{[u]}\>$, so Lemma~\ref{lem:transpositions_strong} implies that $\< A^{(k)}_n, \delta'_{[u]} \> = V_n'$, noting that $\delta'_{[u]}$, being an $(n+2)$-cycle, is an even permutation. Moreover, $A^{(k)}_n \leq A^{(k+1)}_n \leq V_n'$, so $\< A^{(k+1)}_n, \delta'_{[u]} \> = V_n'$. This completes the proof. 
\end{proof}

We will now use Proposition~\ref{prop:delta_symmetric} to obtain two generating sets for $V_n$ and $V_n'$ that will play important roles in the proof of Theorem~\ref{thm:higman}. We begin with an infinite generating set reflecting the self-similar nature of $V_n$.

\begin{proposition} \label{prop:links}
Let $A = \{ x_1, \dots, x_\ell \}$ be a basis. Then
\begin{enumerate}
\item $V_n  = \< V_{n, [x_1]},  \{ (x_10 \ \ x_i) \mid 1 < i \leq \ell \} \>$
\item $V_n' = \< V'_{n, [x_1]}, \{ (x_10 \ \ x_11 \ \ x_i) \mid 1 < i \leq \ell \} \>$.
\end{enumerate}
\end{proposition}

\begin{proof}
Let 
\[
k = \max\{ |x_i| \mid 1 \leq i \leq \ell \}+1.
\] 

We begin with (i). Write
\[
H = \<V_{n, [x_1]}, \{ (x_10 \ \ x_i) \mid 1 < i \leq \ell \} \>.
\] 
Let $v \in X_n$ such that $x_10 \perp v$ and $|v| \geq k$. We will prove that $(x_10 \ \ v) \in H$. Write $v = x_iw$ where $1 \leq i \leq \ell$, noting that $|w| \geq 1$. If $i=1$, then 
\[
(x_10 \ \ v) = (x_10 \ \ x_1w) \in V_{n, [x_1]} \leq H.
\] 
If $i > 1$, then 
\[
(x_10 \ \ v) = (x_10 \ \ x_iw) = (x_10 \ \ x_11w)^{(x_11 \ x_i)} \in H,
\]
since
\[
(x_11 \ \ x_i) = (x_10 \ \ x_i)^{(x_10 \ \ x_11)} \in H.
\]
Therefore, by Lemma~\ref{lem:transpositions_strong}(i), $H = V_n$.

We now assume that $n$ is odd and turn to (ii). Write
\[
K = \<V'_{n, [x_1]}, \{ (x_10 \ \ x_11 \ \ x_i) \mid 1 < i \leq \ell \} \>.
\]
Let $v \in X_n$ such that $x_10 \perp v$, $x_11 \perp v$ and $|v| \geq k$. We will prove that $(x_10 \ \ x_11 \ \ v) \in K$. As above, write $v = x_iw$ where $1 \leq i \leq \ell$. If $i=1$, then 
\[
(x_10 \ \ x_11 \ \ v) = (x_10 \ \ x_11 \ \ x_1w) \in V'_{n, [x_1]} \leq K.
\] 
Now assume that $i > 1$. Write $w=w_0w_1$ where $w_0 \in [n]$ and fix $a \in [n] \setminus \{w_0\}$. Then
\[
(x_10 \ \ x_11 \ \ v) = (x_10 \ \ x_11 \ \ x_iw_0w_1) = (x_10 \ \ x_11 \ \ x_12w_1)^{(x_12 \ \ x_iw_0 \ \ x_ia)} \in H
\]
since
\[
(x_12 \ \ x_1w_0 \ \ x_ia) = (x_12 \ \ x_10w_0 \ \ x_10a)^{(x_10 \ \ x_11 \ \ x_i)^2} \in H.
\]
Therefore, by Lemma~\ref{lem:transpositions_strong}(ii), $K=V_n'$.
\end{proof}

We now turn to generating pairs. Let us introduce some notation for particular types of elements that we will often use.

\begin{notation} \label{not:sigma_tau}
Let $d \geq 1$ and $k \geq \lceil{\log_n{d}}\rceil$, where $\lceil\cdot\rceil$ is the usual ceiling function. Define  
\[
\sigma_n(d,k) = (a_0 \ \ a_1 \ \ \cdots \ \ a_{d-1}) \in S^{(k)}_n \quad \text{and} \quad \tau_n(d,k) = (a_{N-d+1} \ \ \cdots \ \ a_{N-1} \ \ a_{N}) \in S^{(k)}_n
\]
where $N = n^k-1$ and $a_i$ is the $n$\=/ary expansion of $i$ with exactly $k$ digits. We will write 
\[
\sigma_n(d) = \sigma_n(d,\lceil{\log_n{d}}\rceil) \quad \text{and} \quad \tau_n(d) = \tau_n(d, \lceil{\log_n{d}}\rceil).
\]
Notice that $\sigma_n(d,k)$ and $\tau_n(d,k)$ are $d$\=/cycles. For example, the element $\beta \in V$ from Figure~\ref{fig:tree_pairs} is 
\[ 
\tau_2(7) = \tau_2(7,3) = (001 \ \ 010 \ \ 011 \ \ 100 \ \ 101 \ \ 110 \ \ 111).
\]
\end{notation}

Fix primes $p$ and $q$ satisfying 
\begin{equation} \label{eq:primes}
\tfrac{1}{4}n^3 \leq p < \tfrac{1}{2}n^3 \quad \text{and} \quad \tfrac{3}{4}n^3 < q < n^3.
\end{equation}
If $n \geq 5$, then Theorem~\ref{thm:bertand} guarantees that this is possible and it is straightforward to find suitable primes when $1 \leq r < n \leq 4$. When $n=2$ we insist that $p=2$. (The primes $p$ and $q$ depend on $n$, but for clarity of notation we do not give them a subscript of $n$.) Let 
\begin{equation} \label{eq:zeta_beta}
\zeta = \sigma_n(p,3) \quad \text{and} \quad \beta = \tau_n(q,3). 
\end{equation}

Lemma~\ref{lem:permutations_fact} has the following consequence.

\begin{corollary} \label{cor:permutations_fact}
$\<\zeta,\beta\> = A^{(3)}_n$
\end{corollary}

Let us fix two further elements. Recall the definition of $\delta$ and $\delta'$ in \eqref{eq:delta}. Write
\begin{equation} \label{eq:alpha}
\alpha=\delta_{[n-1]}\zeta \quad \text{and} \quad \alpha'=\delta'_{[n-1]}\zeta. 
\end{equation}

We can now state our preferred generating pairs for the Higman--Thompson groups.

\begin{proposition} \label{prop:alpha_beta} \quad
\begin{enumerate}
\item $V_n = \<\alpha,\beta\>$
\item $V_n' = \<\alpha',\beta\>$ if $n$ is odd.
\end{enumerate} 
\end{proposition} 

\begin{proof}
Since $\zeta$ and $\delta_{[n-1]}$ have disjoint support, these two elements commute. If $n > 2$, then $|\zeta|$ is a prime number at least $\frac{1}{4}n^3 > n+1 = |\delta_{[n-1]}|$, so $|\zeta|$ and $|\delta_{[n-1]}|$ are coprime. When $n=2$, we see that $|\zeta| = 2$ and $|\delta_{[1]}|=3$, which again are coprime. Therefore, $\<\alpha\> = \<\zeta, \delta_{[n-1]}\>$. Consequently, Corollary~\ref{cor:permutations_fact} and Lemma~\ref{prop:delta_symmetric}(i) imply
\[
\<\alpha,\beta\> = \<\delta_{[n-1]}, \zeta, \beta\> = \<\delta_{[n-1]}, A^{(3)}_n\> = V_n.
\]

Similarly, if $n$ is odd, then $\<\alpha'\> = \<\zeta,\delta'_{[n-1]}\>$ and using Lemma~\ref{prop:delta_symmetric}(ii) we see that
\[
\<\alpha',\beta\> = \<\delta'_{[n-1]}, \zeta, \beta\> = \<\delta'_{[n-1]}, A^{(3)}_n\> = V_n'. \qedhere
\]
\end{proof}

\begin{remark}\label{rem:alpha_beta}
Let us comment on the generating pairs in Proposition~\ref{prop:alpha_beta}. First note that $|\alpha|$ and $|\alpha'|$ are coprime to $|\beta|$ since $|\beta|$ is a prime number strictly greater than $\frac{3}{4}n^3$ and all the prime factors of $\alpha$ are at most $\frac{1}{2}n^3$. Next note that each of $\alpha$, $\alpha'$ and $\beta$ fix a point in $X^{(3)}_n$. Finally, observe that when $n=2$, $\alpha$ and $\beta$ are the elements from Figure~\ref{fig:tree_pairs}.
\end{remark}

\subsection{Brin--Thompson groups} \label{ss:generators_brin}

This section very closely mirrors the previous. Fix $m \geq 1$. Let us comment on our notation for $mV$. We adopt the convention of Quick \cite{ref:Quick19}, that for $a \in \{0,1\}$, $\boldsymbol{a}$ denotes $(a,a,\dots,a) \in X^m$ and $\boldsymbol{a}_{\boldsymbol{i}}$ denotes $(\varepsilon,\dots,a,\dots,\varepsilon) \in X^m$, where the $a$ is in the $i$th position (and $\varepsilon$ is the empty word). If $u,v \in X^m$, then $u.v$ denotes componentwise concatenation. (The dot in this notation is necessary to distinguish otherwise ambiguous expressions; for example, in $2V$ we write $\boldsymbol{0}.\boldsymbol{0}_1 = (00,0)$ and $\boldsymbol{00}_1 = (00,\varepsilon)$.)

Brin proved the following in \cite{ref:Brin10}. 

\begin{lemma} \label{lem:transpositions_nv}
The group $mV$ is generated by the set of transpositions.
\end{lemma}

By arguing as in the proof of Lemma~\ref{lem:transpositions_strong}(i) and Proposition~\ref{prop:links}(i), we obtain the following; we omit the details.

\begin{lemma} \label{lem:transpositions_strong_nv}
Let $u_0 \in X^m$ with $|u_0| = k \geq 2$. Then $mV = \< (u_0 \ \ v) \mid \text{$u_0 \perp v$ and $|v| \geq k$}\>$. 
\end{lemma}

\begin{proposition} \label{prop:links_nv}
Let $\{ u_1, \dots, u_\ell \}$ be a basis. Then $mV  = \< mV_{[u_1]},  \{ (u_1.\boldsymbol{0} \ \ u_i) \mid 1 < i \leq \ell \} \>$.
\end{proposition}

Let $X^{m,(k)} = \{ (x_1,\dots,x_m) \in X^m \mid |x_i| = k \}$, and write $mS^{(k)} = \mathrm{Sym}(X^{m,(k)})$. Define
\begin{equation} \label{eq:delta_nv}
\delta  = \prod_{i=1}^{m} (\boldsymbol{0} \ \ \boldsymbol{1}.\boldsymbol{0}_{\boldsymbol{i}})(\boldsymbol{0} \ \ \boldsymbol{1}.\boldsymbol{1}_{\boldsymbol{i}}).
\end{equation}

\begin{lemma} \label{lem:delta_symmetric_nv}
$mV  = \< \delta_{[\boldsymbol{1}]}, mS^{(2)} \>$
\end{lemma}

\begin{proof}
If $m \geq 2$, then Quick proves this during \cite[Theorem~1.2]{ref:Quick19}, see the opening paragraphs of \cite[Section~3]{ref:Quick19}. If $m=1$, then $1V = V_2$ and this is Proposition~\ref{prop:delta_symmetric}(i) with $n=2$.
\end{proof}

We introduce notation similar to Notation~\ref{not:sigma_tau}

\begin{notation} \label{not:sigma_tau_nv}
Let $d \geq 1$ and $k \geq \lceil{\log_{2^m}{d}}\rceil$. Let 
\[
\sigma^m(d,k) = (a_0 \ \ a_1 \ \ \cdots \ \ a_{d-1}) \in mS^{(k)} \quad \text{and} \quad \tau^m(d,k) = (a_{N-d+1} \ \ \cdots \ \ a_{N-1} \ \ a_{N}) \in mS^{(k)}
\]
where $N = 2^{mk}-1$ and $a_i$ is the binary expansion of $i$ with exactly $mk$ digits, which we consider an $m$-tuple of length $k$ binary words under the correspondence $y_0 \cdots y_{mk-1} \mapsto (y_0 \cdots y_{k-1}, \dots, y_{(m-1)k} \cdots y_{mk-1})$.  We will write 
\[
\sigma^m(d) = \sigma^m(d,\lceil{\log_{2^m}{d}}\rceil) \quad \text{and} \quad \tau^m(d) = \tau^m(d,\lceil{\log_{2^m}{d}}\rceil).
\]
(Notice that $\sigma^1(d)$ and $\tau^1(d)$ are exactly the elements $\sigma_2(d)$ and $\tau_2(d)$ from Notation~\ref{not:sigma_tau}.)
\end{notation}

By Theorem~\ref{thm:bertand}, we may fix a prime number $q$ and elements $\zeta,\beta \in mV$ satisfying
\begin{equation}\label{eq:primes_zeta_beta_nv}
\tfrac{3}{4}8^m < q < 8^m \quad \text{and} \quad \zeta = \sigma^m(\tfrac{1}{4}\,8^m,3) \quad \text{and} \quad \beta = \tau^m(q,3).
\end{equation}

The next result follows from Lemma~\ref{lem:permutations_fact}, noting that $\beta$ is an odd permutation in $mS^{(3)}$.

\begin{corollary} \label{cor:permutations_fact_nv}
$mS^{(3)} = \<\zeta,\beta\>$
\end{corollary}

To give our preferred generating pair of the Brin--Thompson groups, we write
\begin{equation} \label{eq:alpha_nv}
\alpha=\delta_{[\boldsymbol{1}]}\zeta.
\end{equation}

\begin{proposition} \label{prop:alpha_beta_nv} 
$mV = \<\alpha,\beta\>$
\end{proposition} 

\begin{proof}
Note that $\zeta$ and $\delta_{[\boldsymbol{1}]}$ commute and their orders are coprime ($|\zeta|=\frac{1}{4}8^m$ is a power of two and $|\delta_{[\boldsymbol{1}]}|=2m+1$ is odd). Therefore, $\<\alpha\> = \<\zeta, \delta_{[\boldsymbol{1}]}\>$, so Corollary~\ref{cor:permutations_fact_nv} and Lemma~\ref{lem:delta_symmetric_nv} imply that
\[
\<\alpha,\beta\> = \<\delta_{[\boldsymbol{1}]}, \zeta, \beta\> = \<\delta_{[\boldsymbol{1}]}, mS^{(3)}\> = mV. \qedhere
\]
\end{proof}

\begin{remark}\label{rem:alpha_beta_nv}
Observe that the $\alpha$ and $\beta$ have coprime order and each have a fixed point in $mX^{(3)}$. When $m=1$, these are exactly the elements from Figure~\ref{fig:tree_pairs}.
\end{remark}

\section{Proofs}\label{s:proofs}

In this section, we will prove Theorems~\ref{thm:higman} and~\ref{thm:brin}. In the proof of Theorems~\ref{thm:higman} and~\ref{thm:brin} we will fix the group $G$ as one of $V_n$, $V_n'$ or $mV$ and the element $g$ as a nontrivial element of $G$, then will construct an element $h$, which we prove satisfies $\<g,h\> = G$. To help the reader visualise the element $h$ that we will construct in this proof, we begin this section with an example where, in a special case, we construct this element $h$, depict it via a tree pair and outline the general idea of why $\<g,h\>=G$.

\begin{example}\label{ex:v}
Let $g = (00 \ \ 01) \in V$. Let $x = \alpha_{[00]}$ and $y = \beta_{[01]}$. We have given the tree pairs for $g$ and the product $xy$ in Figures~\ref{fig:g} and~\ref{fig:xy}. Since the subgroup $\<g,xy\>$ stabilises the partition $\{ 00\mathfrak{C}, 01\mathfrak{C}, 1\mathfrak{C} \}$ we deduce that $\<g,xy\>$ is not the entire group $V$. The key idea is to modify $xy$ to obtain an element $h$ such that $\<g,h\> = V$.

For long binary sequences, we adopt notation such as $0^31^20^4$ for $000110000$. Let
\begin{gather*}
z_0 = (00 \ \ 01 \ \ 10 \ \ 11)_{[0^310]} \cdot (0^310^3 \ \ 010^3) \\[3pt]
z_1 = (0000 \ \ 0001 \ \ \cdots \ \ 1010)_{[0^31^2]} \cdot (0^31^20^4 \ \ 1)
\end{gather*}
and define the element $h = xyz_0z_1$, which is depicted in Figure~\ref{fig:h}. Note that the leaves in the tree pair for $h$ that are moved by $x$, $y$, $z_0$ and $z_1$ are labelled by $a_i$, $b_i$, $c_i$ and $d_i$, respectively.

Let us make some observations about the element $h$. First note that $x$, $y$, $z_0$ and $z_1$ have orders 6, 7, 5 and 11, respectively. In particular, these elements have coprime orders. Moreover, these elements have disjoint support and therefore commute. Consequently, each of these elements is a suitable power of $h$ and, accordingly, is contained in $\<g, h\>$. 

We now outline the argument for why $\<g,h\> = V$. Proposition~\ref{prop:alpha_beta} implies that 
\[
V_{[00]} = \<\alpha_{[00]},\beta_{[00]}\> = \<x, y^g\> \leq \<g,h\>.
\]
In addition, $V_{[01]} = (V_{[00]})^g \leq \<g,h\>$. Using appropriate elements from $V_{[00]}$ and $V_{[01]}$ we can show that
\[
(000 \ \ 01) \in \< V_{[00]}, V_{[01]}, z_0 \> \leq \<g,h\> \text{ \ \ and \ \ } (000 \ \ 1) \in \< V_{[00]}, z_1 \> \leq \<g,h\>.
\]
This shows that 
\[
\< V_{[00]}, (000 \ \ 01), (000 \ \ 1) \> \leq \<g,h\>,
\]
so Proposition~\ref{prop:links} implies that $\<g,h\> = V$.
\end{example}

\begin{figure}[p]
\[ 
  \begin{tikzpicture}[
        inner sep=0pt,
        baseline=-30pt,
        level distance=20pt,
        level 1/.style={sibling distance=60pt},
        level 2/.style={sibling distance=30pt}
      ]
      \node (root) [circle,fill] {}
      child {node (0) [circle,fill] {}
        child {node (00) {$1$}}
        child {node (01) {$2$}}}
      child {node (1) {}};
  \end{tikzpicture}
    \to
  \begin{tikzpicture}[
        inner sep=0pt,
        baseline=-30pt,
        level distance=20pt,
        level 1/.style={sibling distance=60pt},
        level 2/.style={sibling distance=30pt}
      ]
      \node (root) [circle,fill] {}
      child {node (0) [circle,fill] {}
        child {node (00) {$2$}}
        child {node (01) {$1$}}}
      child {node (1) {}};
    \end{tikzpicture}
  \]
\caption{Tree pair for $g = (00 \ \ 01) \in V$} \label{fig:g}
\end{figure}

\begin{figure}[p]
\[
  \begin{tikzpicture}[
      scale=0.9,
      inner sep=0pt,
      baseline=-30pt,
      level distance=20pt,
      level 1/.style={sibling distance=120pt}
    ]
    \node (root) [circle,fill] {}
    child {[sibling distance=120pt] node (0) [circle,fill] {}
      child {[sibling distance=60pt] node (00) [circle,fill] {}
        child {[sibling distance=30pt] node (000) [circle,fill] {}
          child {[sibling distance=15pt] node (0000) [circle, fill] {}
            child {node (00000) {$a_1$}}
            child {node (00001) {$a_2$}}}
          child {[sibling distance=15pt] node (0001) {}}}
        child {[sibling distance=30pt] node (001) [circle, fill] {}
          child {[sibling distance=15pt] node (0010) {$a_3$}}
          child {[sibling distance=15pt] node (0011) [circle, fill] {}
            child {node (00110) {$a_4$}}
            child {node (00111) {$a_5$}}}}} 
      child {[sibling distance=60pt] node (01) [circle, fill] {} 
        child {[sibling distance=30pt] node (010) [circle, fill] {}
          child {[sibling distance=15pt] node (0100) [circle, fill] {}
            child {node (01000) {}}
            child {node (01001) {$b_1$}}}
          child {[sibling distance=15pt] node (0101) [circle,fill] {}
            child {node (01010) {$b_2$}}
            child {node (01011) {$b_3$}}}}
        child {[sibling distance=30pt] node (011) [circle, fill] {}
          child {[sibling distance=15pt] node (0110) [circle, fill] {}  
            child {node (01100) {$b_4$}}  
            child {node (01101) {$b_5$}}}
          child {[sibling distance=15pt] node (0111) [circle, fill] {}   
            child {node (01110) {$b_6$}}  
            child {node (01111) {$b_7$}}}}}}           
    child {node (1) {}};
  \end{tikzpicture}
\quad \to \!\!\!\!\!\!\!\!
  \begin{tikzpicture}[
      scale=0.9,
      inner sep=0pt,
      baseline=-30pt,
      level distance=20pt,
      level 1/.style={sibling distance=120pt}
    ]
    \node (root) [circle,fill] {}
    child {[sibling distance=120pt] node (0) [circle,fill] {}
      child {[sibling distance=60pt] node (00) [circle,fill] {}
        child {[sibling distance=30pt] node (000) [circle,fill] {}
          child {[sibling distance=15pt] node (0000) [circle, fill] {}
            child {node (00000) {$a_2$}}
            child {node (00001) {$a_1$}}}
          child {[sibling distance=15pt] node (0001) {}}}
        child {[sibling distance=30pt] node (001) [circle, fill] {}
          child {[sibling distance=15pt] node (0010) {$a_5$}}
          child {[sibling distance=15pt] node (0011) [circle, fill] {}
            child {node (00110) {$a_3$}}
            child {node (00111) {$a_4$}}}}} 
      child {[sibling distance=60pt] node (01) [circle, fill] {} 
        child {[sibling distance=30pt] node (010) [circle, fill] {}
          child {[sibling distance=15pt] node (0100) [circle, fill] {}
            child {node (01000) {}}
            child {node (01001) {$b_7$}}}
          child {[sibling distance=15pt] node (0101) [circle,fill] {}
            child {node (01010) {$b_1$}}
            child {node (01011) {$b_2$}}}}
        child {[sibling distance=30pt] node (011) [circle, fill] {}
          child {[sibling distance=15pt] node (0110) [circle, fill] {}  
            child {node (01100) {$b_3$}}  
            child {node (01101) {$b_4$}}}
          child {[sibling distance=15pt] node (0111) [circle, fill] {}   
            child {node (01110) {$b_5$}}  
            child {node (01111) {$b_6$}}}}}}           
    child {node (1) {}};
  \end{tikzpicture}
\]
\caption{Tree pair for $xy = \alpha_{[00]}\beta_{[01]} \in V$ from Example~\ref{ex:v}}\label{fig:xy}
\end{figure}

\begin{figure}[p]
\centering
  \begin{tikzpicture}[
      scale=0.9,
      inner sep=0pt,
      baseline=-50pt,
      level distance=20pt,
      level 1/.style={sibling distance=160pt},
      level 2/.style={sibling distance=160pt},
      level 10/.style={sibling distance=15pt}
    ]
    \node (root) [circle,fill] {}
    child {[sibling distance=160pt] node (0) [circle,fill] {}
      child {[sibling distance=120pt] node (00) [circle,fill] {}
        child {[sibling distance=90pt] node (000) [circle,fill] {}
          child {[sibling distance=15pt] node (0000) [circle, fill] {}
            child {node (00000) {$a_1$}}
            child {node (00001) {$a_2$}}}
          child {[sibling distance=120pt] node (0001) [circle, fill] {}
              child {[sibling distance = 30pt] node (00010) [circle, fill] {}
                child {[sibling distance = 15pt] node (000100) [circle, fill] {}
                  child {node (0001000) {$c_1$}}
                  child {node (0001001) {$c_2$}}}
                child {[sibling distance = 15pt] node (000101) [circle, fill] {}
                  child {node (0001010) {$c_3$}}
                  child {node (0001011) {$c_4$}}}}
              child {[sibling distance=95pt] node (00011) [circle, fill] {} 
                child {[sibling distance=60pt] node (000110) [circle, fill] {}
                  child {[sibling distance=30pt] node (0001100) [circle, fill] {} 
                    child {[sibling distance=15pt] node (00011000) [circle, fill] {}
                      child {node (000110000) {$d_1$} }
                      child {node (000110001) {$d_2$} }}
                    child {[sibling distance=15pt] node (00011001) {} 
                      child {node (000110010) {$d_3$} }
                      child {node (000110011) {$d_4$} }}} 
                  child {[sibling distance=30pt] node (0001101) [circle, fill] {} 
                    child {[sibling distance=15pt] node (00011010) [circle, fill] {} 
                      child {node (000110100) {$d_5$}}
                      child {node (000110101) {$d_6$}}}
                    child {[sibling distance=15pt] node (00011011) [circle, fill] {}
                      child {node (000110110) {$d_7$}}
                      child {node (000110111) {$d_8$}}}}}
                child {[sibling distance=15pt] node (000111) [circle, fill] {}
                  child {[sibling distance=15pt] node (0001110) [circle, fill] {}
                    child {[sibling distance=15pt] node (00011100) [circle, fill] {}
                      child {node (000111000) {$d_9$} }
                      child {node (000111001) {$d_{10}$} }}
                    child {node (00011101) {} }} 
                  child {[sibling distance=15pt] node (0001111) {}}}}}}
        child {[sibling distance=30pt] node (001) [circle, fill] {}
          child {[sibling distance=15pt] node (0010) {$a_3$}}
          child {[sibling distance=15pt] node (0011) [circle, fill] {}
            child {node (00110) {$a_4$}}
            child {node (00111) {$a_5$}}}}} 
      child {[sibling distance=60pt] node (01) [circle, fill] {} 
        child {[sibling distance=30pt] node (010) [circle, fill] {}
          child {[sibling distance=15pt] node (0100) [circle, fill] {}
            child {node (01000) {$c_5$}}
            child {node (01001) {$b_1$}}}
          child {[sibling distance=15pt] node (0101) [circle,fill] {}
            child {node (01010) {$b_2$}}
            child {node (01011) {$b_3$}}}}
        child {[sibling distance=30pt] node (011) [circle, fill] {}
          child {[sibling distance=15pt] node (0110) [circle, fill] {}  
            child {node (01100) {$b_4$}}  
            child {node (01101) {$b_5$}}}
          child {[sibling distance=15pt] node (0111) [circle, fill] {}   
            child {node (01110) {$b_6$}}  
            child {node (01111) {$b_7$}}}}}}           
    child {node (1)  {$d_{11}$}};
  \end{tikzpicture}
  \vspace{11pt}
  \[
  \to
  \]
  \begin{tikzpicture}[
      scale=0.9,
      inner sep=0pt,
      baseline=-50pt,
      level distance=20pt,
      level 1/.style={sibling distance=160pt},
      level 2/.style={sibling distance=160pt},
      level 10/.style={sibling distance=15pt}
    ]
    \node (root) [circle,fill] {}
    child {[sibling distance=160pt] node (0) [circle,fill] {}
      child {[sibling distance=120pt] node (00) [circle,fill] {}
        child {[sibling distance=90pt] node (000) [circle,fill] {}
          child {[sibling distance=15pt] node (0000) [circle, fill] {}
            child {node (00000) {$a_2$}}
            child {node (00001) {$a_1$}}}
          child {[sibling distance=120pt] node (0001) [circle, fill] {}
              child {[sibling distance = 30pt] node (00010) [circle, fill] {}
                child {[sibling distance = 15pt] node (000100) [circle, fill] {}
                  child {node (0001000) {$c_5$}}
                  child {node (0001001) {$c_1$}}}
                child {[sibling distance = 15pt] node (000101) [circle, fill] {}
                  child {node (0001010) {$c_2$}}
                  child {node (0001011) {$c_3$}}}}
              child {[sibling distance=95pt] node (00011) [circle, fill] {} 
                child {[sibling distance=60pt] node (000110) [circle, fill] {}
                  child {[sibling distance=30pt] node (0001100) [circle, fill] {} 
                    child {[sibling distance=15pt] node (00011000) [circle, fill] {}
                      child {node (000110000) {$d_{11}$} }
                      child {node (000110001) {$d_1$} }}
                    child {[sibling distance=15pt] node (00011001) {} 
                      child {node (000110010) {$d_2$} }
                      child {node (000110011) {$d_3$} }}} 
                  child {[sibling distance=30pt] node (0001101) [circle, fill] {} 
                    child {[sibling distance=15pt] node (00011010) [circle, fill] {} 
                      child {node (000110100) {$d_4$}}
                      child {node (000110101) {$d_5$}}}
                    child {[sibling distance=15pt] node (00011011) [circle, fill] {}
                      child {node (000110110) {$d_6$}}
                      child {node (000110111) {$d_7$}}}}}
                child {[sibling distance=15pt] node (000111) [circle, fill] {}
                  child {[sibling distance=15pt] node (0001110) [circle, fill] {}
                    child {[sibling distance=15pt] node (00011100) [circle, fill] {}
                      child {node (000111000) {$d_8$} }
                      child {node (000111001) {$d_9$} }}
                    child {node (00011101) {} }} 
                  child {[sibling distance=15pt] node (0001111) {}}}}}}
        child {[sibling distance=30pt] node (001) [circle, fill] {}
          child {[sibling distance=15pt] node (0010) {$a_5$}}
          child {[sibling distance=15pt] node (0011) [circle, fill] {}
            child {node (00110) {$a_3$}}
            child {node (00111) {$a_4$}}}}} 
      child {[sibling distance=60pt] node (01) [circle, fill] {} 
        child {[sibling distance=30pt] node (010) [circle, fill] {}
          child {[sibling distance=15pt] node (0100) [circle, fill] {}
            child {node (01000) {$c_4$}}
            child {node (01001) {$b_7$}}}
          child {[sibling distance=15pt] node (0101) [circle,fill] {}
            child {node (01010) {$b_1$}}
            child {node (01011) {$b_2$}}}}
        child {[sibling distance=30pt] node (011) [circle, fill] {}
          child {[sibling distance=15pt] node (0110) [circle, fill] {}  
            child {node (01100) {$b_3$}}  
            child {node (01101) {$b_4$}}}
          child {[sibling distance=15pt] node (0111) [circle, fill] {}   
            child {node (01110) {$b_5$}}  
            child {node (01111) {$b_6$}}}}}}           
    child {node (1)  {$d_{10}$}};
  \end{tikzpicture}
\vspace{5pt}
\caption{Tree pair for $h=xyz_0z_1 \in V$ from Example~\ref{ex:v}}\label{fig:h}
\end{figure}

We now prove Theorems~\ref{thm:higman} and~\ref{thm:brin}.

\begin{proof}[Proof of Theorems~\ref{thm:higman} and~\ref{thm:brin}]
Let $(G,\Gamma,\Omega)$ be one of $V_n$, $V_n'$ or $mV$. Let $g \in G$ be nontrivial. If $G$ is $V_n$ or $V_n'$, then let $(\Gamma,\Omega) = (\mathfrak{C}_n,X_n)$, and if $G$ is $mV$, then $(\Gamma,\Omega) = (\mathfrak{C}^m,X^m)$. By Lemma~\ref{lem:incomparables_exist}, we may fix $u \in \Omega$ such that $u \perp ug$. By Lemma~\ref{lem:incomparables_extend}, we may fix $w_1, \dots, w_\ell \in \Omega$ such that $\mathcal{W} = \{u, ug, w_1, \dots, w_\ell\}$ is a basis for $\Gamma$. For clarity of notation, write $v=ug$.

We will now define an element $h \in G$, which we will prove satisfies $\<g, h \> = G$. Bear in mind that the definition of $h$ in Example~\ref{ex:v} is a special case of the construction we are about to describe.

Let $y = \beta_{[v]}$. Let $a,b \in \Omega^{(3)}$ be fixed points of $\alpha$ and $\beta$ respectively (see Remarks~\ref{rem:alpha_beta} and~\ref{rem:alpha_beta_nv}). Fix $0 \leq i \leq \ell$, recalling that $\ell+2$ is the size of the basis $\mathcal{W}$. We now define an element $z_i$. Write $u_i = uau_i'$ where $u_i'$ is the $n$\=/ary (where $n=2$ if $G=mV$) expansion of $i$ with exactly $\lceil{\log_n(\ell+1)}\rceil$ digits. Let $p_1 < p_2 < p_3 \dots$ be the prime numbers strictly greater than $q$ (recall that we defined the prime $q$ in  \eqref{eq:primes} in the case of $V_n$ and $V_n'$ and in \eqref{eq:primes_zeta_beta_nv} in the case of $mV$). Write $w_0 = vb$ and $p_0=5$. Define 
\[
z_i = \sigma(p_i-1)_{[u_i]} \cdot (u_i0^{\lceil{\log_n{(p_i-1)}}\rceil} \ \ w_i).
\]
where $\sigma$ is $\sigma_n$ or $\sigma^n$ as appropriate. Observe that $z_i$ is a $p_i$-cycle. 

We split into three cases depending on whether $G$ is $V_n$, $V_n'$ or $mV$. The arguments in each of these cases are quite similar.

\noindent\textbf{Case 1. $G = V_n$}\nopagebreak

Let $x = \alpha_{[u]}$ and define 
\[
h = x y z_0 z_1 \cdots z_\ell.
\]
Since $x$, $y$, $z_0$, $z_1$, \dots, $z_\ell$ have pairwise coprime orders and commute, each of these elements is a suitable power of $h$ and is, thus, contained in $\<g,h\>$. 

Proposition~\ref{prop:alpha_beta}(i) implies that 
\[ 
V_{n,[u]} = \< \alpha_{[u]}, \beta_{[u]} \> = \< x, y^g \>  \leq \< g, h \>.
\] 
In addition, $V_{n,[v]} = (V_{n,[u]})^g \leq \< g, h \>$. 

Let $0 \leq i \leq \ell$. Then $\sigma_n(p_i-1)_{[u_i]}$ is contained in $V_{[u_i]} \leq V_{[u]} \leq \<g,h\>$, so
\[
(u0 \ \ w_i) = (u_i0^{\lceil{\log_n{(p_i-1)}}\rceil} \ \ w_i)^{\chi_i} = ((\sigma_n(p_i-1)_{[u_i]})^{-1} \cdot z_i)^{\chi_i} \in \<g,h\>.
\]
where $\chi_i \in V_{[u]} \leq \<g,h\>$ satisfies $u_i0^{\lceil{\log_2{(p_i-1)}}\rceil}\chi_i = u_i0$. 

Recall that $w_0=vb$. For $c \in [n]$,
\[
(u0c \ \ vc) = (u0 \ \ w_0)^{\phi_c\psi_c} \in \<g,h\>
\]
where $\phi_c \in V_{[u]}$ satisfies $u0\phi_c = u0c$ and $\psi_c \in V_{[v]}$ satisfies $vb\psi_c = vc$. Consequently,
\[
(u0 \ \ v) = \prod_{c=0}^{n-1}(u0c \ \ vc) \in \<g,h\>. 
\]

Therefore, by Proposition~\ref{prop:links}(i),
\[
\<g,h\> = \< V_{n,[u]}, (u0 \ \ v), \{ (u0 \ \ w_i) \mid 1 \leq i \leq \ell \} \> = V_n.
\]

\noindent \textbf{Case 2. $G=V_n'$}\nopagebreak

Here we may, and will, assume that $n$ is odd. Let $x = \alpha'_{[u]}$ and define $h = x y z_0 z_1 \cdots z_\ell$. As in Case~1, $x, y, z_0, z_1, \dots, z_\ell \in \<g,h\>$. Now Proposition~\ref{prop:alpha_beta}(ii) implies that $V_{n,[u]}' = \<\alpha'_{[u]}, \beta_{[u]}\> = \<x,y^{g^{-1}}\> \leq \<g,h\>$, and we also have $V_{n,[v]}' = (V'_{n,[u]})^g \leq \< g, h \>$.

Let $0 \leq i \leq \ell$. Then $\sigma_n(p_i-2)_{[u_i]} \in V_{[u]} \leq \<g,h\>$, so
\[
(u0 \ \ u1 \ \ w_i) = (u_i0^{\lceil{\log_n{(p_i-1)}}\rceil-1}0 \ \  u_i0^{\lceil{\log_n{(p_i-1)}}\rceil-1}1 \ \ w_i)^{\chi_i} = ((\sigma_n(p_i-2)_{[u_i]})^{-1} \cdot z_i)^{\chi_i} \in \<g,h\>.
\]
where $\chi_i \in V'_{[u]} \leq \<g,h\>$ satisfies $u_i0^{\lceil{\log_n{(p_i-1)}}\rceil-1}0\chi_i = u_i0$ and $u_i0^{\lceil{\log_n{(p_i-1)}}\rceil-1}1\chi_i = u_i1$.

Recall that $w_0=vb$. For $c \in [n]$,
\[
(u0c \ \ u1c \ \ vc) = (u0 \ \ u1 \ \ w_0)^{\phi_c\psi_c} \in \<g,h\>
\]
where $\phi_c \in V_{[u]}$ satisfies $u0\phi_c = u0c$ and $u1\phi_c = u1c$ and where $\psi_c \in V_{[v]}$ satisfies $vb\psi_c = vc$. Therefore,
\[
(u0 \ \ u1 \ \ v) = \prod_{c=0}^{n-1}(u0c \ \ u1c \ \ vc) \in \<g,h\>. 
\] 

Consequently, Proposition~\ref{prop:links}(ii) implies that
\[
\<g,h\> = \< V_{n,[u]}', (u0 \ \ u1 \ \ v), \{ (u0 \ \ u1 \ \ w_i) \mid 1 \leq i \leq \ell \} \> = V_n'.
\]

\noindent \textbf{Case 3. $G=mV$}\nopagebreak

Let $x = \alpha_{[u]}$ and $h = x y z_0 z_1 \cdots z_\ell$. As in the previous cases,  $x, y, z_0, z_1, \dots, z_\ell \in \<g,h\>$. In particular, by Proposition~\ref{prop:alpha_beta_nv}, $mV_{[u]} = \<\alpha_{[u]}, \beta_{[u]} \> = \< x, y^{g^{-1}}\> \leq \<g,h\>$. By an argument almost identical to that in Case~1, we see that $(u.\boldsymbol{0} \ \ v) \in \<g,h\>$ and $(u.\boldsymbol{0} \ \ w_i) \in \<g,h\>$ for all $1 \leq i \leq \ell$. Consequently, by Proposition~\ref{prop:links_nv}, 
\[
\<g,h\> = \< mV_{[u]}, (u.\boldsymbol{0} \ \ v), \{ (u.\boldsymbol{0} \ \ w_i) \mid 1 \leq i \leq \ell \} \> = mV,
\] 
which completes the proof. 
\end{proof}

\section{Questions} \label{s:questions}

In this section, we present some open problems, which connect our main results with the existing work on finite groups.

In the introduction we remarked that a finite group is conjectured to be $\frac{3}{2}$\=/generated if and only if every proper quotient is cyclic, and that while this conjecture is false for arbitrary groups, there are no known $2$\=/generated counterexamples. This leads to the following natural question. (The authors were recently made aware that this question has attracted attention on \texttt{MathOverflow} \cite{ref:Palcoux19}.)

\begin{myquestion}
Is there a $2$\=/generated group with no noncyclic proper quotients that is not $\frac{3}{2}$\=/generated?
\end{myquestion}

Following \cite{ref:BrennerWiegold75}, the \emph{spread} of a group $G$, written $s(G)$, is the greatest nonnegative integer $k$ such that for all nontrivial elements $x_1, \dots, x_k$ there exists $y \in G$ such that 
\begin{equation} \label{eq:spread}
\<x_1,y\> = \cdots = \<x_k,y\> = G.
\end{equation}
Write $s(G)=\infty$ if there is no such maximum. Therefore, $s(G) \geq 1$ is equivalent to $G$ being $\frac{3}{2}$\=/generated. It is known that $s(G) \geq 2$ if $G$ is a finite simple group \cite{ref:BreuerGuralnickKantor08}. Indeed, it is conjectured that there does not exist a finite group $G$ satisfying $s(G)=1$ \cite[Conjecture~3.15]{ref:Burness19}. We pose the following questions.

\begin{myquestion} \label{q:spread_one}
Is there an infinite group $G$ such that $s(G)=1$?
\end{myquestion}

\begin{myquestion}
Is there a (necessarily infinite) noncyclic group $G$ such that $s(G) = \infty$.
\end{myquestion}

The \emph{uniform spread} of a group $G$, written $u(G)$, is the greatest nonnegative integer $k$ for which there exists a conjugacy class $C$ of $G$ such that for all nontrivial elements $x_1,\dots,x_k\in G$ there exists $y \in C$ such that \eqref{eq:spread} holds (and write $u(G) = \infty$ if there is no such maximum). Clearly, $u(G) \leq s(G)$. Every nonabelian finite simple group satisfies $u(G) \geq 2$ \cite{ref:BreuerGuralnickKantor08} and the only known nonabelian finite group for which $s(G) > 0$ but $u(G)=0$ is $S_6$ (see \cite[Theorem~2]{ref:BurnessGuest13} for example). It is therefore natural to ask the following.

\begin{myquestion}
Is there an infinite noncyclic group $G$ such that $s(G) > 0$ but $u(G)=0$?
\end{myquestion}

Following \cite{ref:BurnessHarper19}, a subset $S \subseteq G$ is a \emph{total dominating set} for $G$ if for all $x \in G$ there exists $y \in S$ such that $\<x,y\>=G$. Therefore, $u(G) > 0$ if and only if $G$ has a total dominating set consisting of conjugate elements, a so-called \emph{uniform dominating set}.

\begin{myquestion}
Is there an infinite noncyclic group with a finite total dominating set?
\end{myquestion}

Of course, the authors are particularly interested in whether the the Higman--Thompson or Brin--Thompson groups provide affirmative answers to these questions.

\vspace{11pt}

\noindent C. Donoven \newline
Department of Mathematical Sciences, Binghamton University, New York 13902, USA \newline
\texttt{cdonoven@math.binghamton.edu}

\vspace{5pt}

\noindent S. Harper \newline
School of Mathematics, University of Bristol, BS8 1UG, UK \newline
Heilbronn Institute for Mathematical Research, Bristol, UK \newline
\texttt{scott.harper@bristol.ac.uk}

\end{document}